\documentclass[10pt,reqno]{amsart}
\usepackage[utf8]{inputenc}
\usepackage[english]{babel}
\usepackage{amsmath, amssymb, amsthm}
\usepackage{mathrsfs}
\usepackage{enumitem}
\usepackage{graphicx}
\usepackage{tikz-cd}
\usepackage{xcolor}
\numberwithin{equation}{section} 

\usepackage{diagbox} 
\usepackage{tensor}  

\usepackage{hyperref} 

\usepackage[alphabetic]{amsrefs} 

\theoremstyle{plain}
\newtheorem{thm}{Theorem}[section]
\newtheorem{prop}[thm]{Proposition}
\newtheorem{lemma}[thm]{Lemma}

\theoremstyle{definition}
\newtheorem{defin}[thm]{Definition}

\allowdisplaybreaks

\newcommand{\R}{\mathbb{R}}

\newcommand{\eps}{\varepsilon}

\DeclareMathOperator{\diag}{diag}

\def\XXint#1#2#3{{\setbox0=\hbox{$#1{#2#3}{\int}$ }
\vcenter{\hbox{$#2#3$ }}\kern-.6\wd0}}

\usepackage{scalerel,stackengine}
\stackMath
\newcommand\hhat[1]{%
\savestack{\tmpbox}{\stretchto{%
  \scaleto{%
    \scalerel*[\widthof{\ensuremath{#1}}]{\kern.1pt\mathchar"0362\kern.1pt}%
    {\rule{0ex}{\textheight}}
  }{\textheight}%
}{2.4ex}}%
\stackon[-6.9pt]{#1}{\tmpbox}%
}

\title[Boundary rigidty in a fixed conformal class on AHM]{Boundary Rigidity in a fixed conformal class for Asymptotically Hyperbolic Manifolds}

\author{Tristan Humbert}
\address{Sorbonne universite, Paris France 75005.}
\email{humbertt@imj-prg.fr}

\author{Sebastián Muñoz-Thon}
\address{Universit\'e Paris-Saclay, Laboratoire de math\'ematiques d’Orsay, 91405, Orsay, France.}
\email{sebastian.munoz-thon@universite-paris-saclay.fr}

\begin{document}

\begin{abstract}
Given two conformal asymptotically hyperbolic metrics which are either both simple or both negatively curved, we show that if their (marked) renormalized boundary distances coincide for some choices of conformal representatives in their conformal infinities, then the two metrics are equal.
\end{abstract}

\maketitle

\section{Introduction}
\subsection{Main result}

Let $\overline{M}=M \sqcup \partial M$ be a smooth and compact manifold with boundary of dimension $n+1 \geq 2$. We say that a smooth Riemannian metric $g$ on the interior $M$ is \emph{asymptotically hyperbolic} (AH) if there exists a \emph{boundary defining function} $\rho_{0}\in C^{\infty}(\overline{M},[0,\infty)) $ for $\partial M$ (i.e., $\{\rho_0=0\}=\partial M$ and $d\rho_{0}|_{\partial M} \neq 0$), such that $\overline{g}:=\rho_{0}^{2}g$ can be extended smoothly to a Riemannian metric on $\overline{M}$, and $|d\rho_{0}|_{\overline{g}}^{2}=1$ over $\partial M$. 

This notion does not depend on the choice of such boundary defining function. However, the restriction to $\partial M$ of the extension $\rho_{0}^{2} g$ of the metric $g$ depends on the choice of $\rho_{0}$ while its conformal class does not. The boundary $\partial {M}$ equipped with the conformal class of $\rho_{0}^{2} g|_{T \partial {M}}$ is called the \emph{conformal boundary}, or \emph{conformal infinity}, of $g$. An element $h$ in the conformal infinity of $g$ is called a \emph{conformal representative} of~$g$.

Let $\varphi_t \colon S^*M\to S^*M$ denote the geodesic flow on the unit cotangent bundle of $M$. We define the \emph{incoming} and
\emph{outgoing trapped sets} as
\begin{equation}
    \label{eq:trapped}
    \Gamma_{\pm}:=\{(x,\xi)\in S^*M\mid \rho_0(\varphi_t(x,\xi))\not\to 0 \text{ as } t\to \mp \infty\},
\end{equation}
where $\Gamma_+$ (resp. $\Gamma_-$) corresponds to geodesics trapped in the past (resp. in the future). Note that $\Gamma_\pm$ do not depend on the choice of boundary defining function~$\rho_0$. We say that an asymptotically hyperbolic metric $g$ is \emph{non-trapping} if $\Gamma_+=\Gamma_-=\emptyset$. Note that in this case, $\overline M$ is simply connected, see \cite{GGSU19}*{p. 2871}.
\begin{defin}
    An AH  metric $g$ on $M$ is \emph{simple} if it satisfies the following:
    \begin{itemize}
        \item the metric $g$ is non-trapping;
        \item the metric $g$ has no conjugate point at infinity\footnote{See also \cite{GGSU19}*{Definition 5.4} for an equivalent definition using (un)stable bundles.}, i.e., there are no nonzero Jacobi fields that
decay to $0$ as $t\to +\infty$ and $t\to -\infty$ along any geodesic.
    \end{itemize}
    In this case, we say that $(M,g)$ is a simple AH manifold.
\end{defin}
We remark that a non-trapping AH manifold with non-positive sectional curvature is simple, see the discussion after \cite{GGSU19}*{Definition 5.4}.

Observe as well that a simple metric does not have conjugate points, see \cite{GGSU19}*{Propositions 5.5 \& 5.6}. Moreover, given two points $p\neq q\in \partial M$, there exists a unique $g$-geodesic $\gamma_g(p,q)$ joining $p$ and $q$, see \cite{GGSU19}*{Proposition 5.12}.

We define the \emph{renormalized boundary distance function} by
    \begin{equation}
    \label{eq:bdrdistance}
        b_{g} \colon \partial {M} \times \partial {M} \setminus \diag  \to \mathbb R,\quad
        (p,q) \mapsto  L_{g}(\gamma_g(p,q)),
    \end{equation}
    where $L_g$ denotes the \emph{renormalized length}, see Definition \ref{defi:length}.
    Note that $b_g$ and $L_{g}$ depend on the choice of boundary defining function $\rho_0$, or equivalently, on the choice of representative in the conformal infinity.
    The first result of this paper shows that the renormalized boundary distance function characterizes simple metrics within their conformal class.
\begin{thm}
\label{theo:main}
    Let $(\overline{M},g)$ be a simple AH manifold. Let $c \in C^{\infty}(\overline{M},\mathbb R_+)$ be such that $(\overline{M},c^2g)$ is also a simple AH manifold. Assume that for some choices $h$ and $h'$ of conformal representatives in the conformal infinities of $g$ and $c^2g$, the renormalized boundary distances are equal. Then $c \equiv 1$, i.e., the metrics are equal.
\end{thm}
For our next result, let $(\overline M,g)$ be a negatively curved AH metric. Given two points $p\neq q\in \partial M$, let $\mathcal P_{p,q}$ be the set of homotopy classes of curves joining $p$ and $q$. For any $[\gamma]\in \mathcal P_{p,q}$, there is a unique geodesic $\gamma_g(p,q,[\gamma])$ in $[\gamma]$ joining $p$ and $q$. We define
$$ \mathcal D:=\{(p,q,[\gamma]): (p,q)\in \partial {M} \times \partial {M} \setminus \diag , \ [\gamma]\in \mathcal P_{p,q}\}.$$
The \emph{renormalized marked boundary
distance function} is 
   \begin{equation}
    \label{eq:markedbdrdistance}
        b_{g} \colon \mathcal D  \to \mathbb R,\quad
        (p,q,[\gamma]) \mapsto  L_{g}(\gamma_g(p,q,[\gamma])).
    \end{equation}
 Our second result shows that the  renormalized marked boundary
distance characterizes negatively curved AH metrics within their conformal class. 
\begin{thm}
\label{theo:main2}
    Let $(\overline{M},g)$ be a negatively curved AH manifold. Let $c \in C^{\infty}(\overline{M},\mathbb R_+)$ be such that $(\overline{M},c^2g)$ is also a negatively curved AH manifold. Assume that for some choices $h$ and $h'$ of conformal representatives in the conformal infinities of $g$ and $c^2g$, the renormalized marked boundary distances are equal. Then $c \equiv 1$, i.e., the metrics are equal.
\end{thm}

\subsection{Related results} On compact manifolds with boundary, Michel's conjecture asks if the distance between boundary points determines the metric, modulo isometries fixing the boundary pointwise \cite{Michel8}. When the manifold is simple (see \cite{PSU23}*{\S 3.8} for the definition of a simple compact manifold), the problem has been completely solved for surfaces \cite{PU05}, and in higher dimensions under a foliation condition \cite{SUV21}. For metrics in the same conformal class, the result was proved in \cite{Muhometov81}, although the proof in \cite{Croke91} is closer to our approach. The rigidity of the marked boundary distance function on compact negatively curved (or more generally Anosov) surfaces was proven in \cites{GM18,EL24}.

Regarding the problem on AH manifolds, a similar problem can be studied using the renormalized boundary distance function. Let us cite \cite{GGSU19}, where the authors prove a deformation boundary rigidity result, besides studying the X-ray transform. We note that the boundary rigidity problem for two negatively curved AH surfaces was solved in \cite{Lef20}. In a similar spirit as in \cite{SUV21}, the X-ray transform in AH geometries was studied in \cites{EG21, Eptaminitakis22}. For the case of hyperbolic manifolds, this operator has been extensively studied, we refer to the recent paper \cite{EMZ26} and the references therein. We also mention \cite{MK25}, where a version of Michel's conjecture was studied using renormalized areas of minimal surfaces.

The results of this paper fit into the asymptotically hyperbolic boundary rigidity program in a complementary way. Theorem~\ref{theo:main} can be viewed as the AH analogue of the classical conformal boundary rigidity result of Mukhometov--Croke for compact simple manifolds, with the usual boundary distance replaced by the renormalized boundary distance. This gives a global rigidity result for metrics in the same conformal class, improving \cite{GGSU19}*{Theorem~1.5} in this context. Likewise, Theorem~\ref{theo:main2} should be compared with \cite{GM18}*{Lemma 2.5} in the compact setting and gives the corresponding marked rigidity statement for negatively curved AH metrics, extending \cite{Lef20}*{Theorem~1.2} to any dimension when the metrics are in the same conformal class.

Finally, we mention that geometric inverse problems of this kind have been studied in other degenerated geometries, such as asymptotically conic ones \cites{GLT22, VZ24, JV24}, and manifolds with hyperbolic cups \cites{GBL23I, GBL23II}.

\subsection{Strategy of the proof}
The strategy of the proofs of Theorem \ref{theo:main} and \ref{theo:main2} follows in spirit the proof of Croke \cite{Croke91} in the compact case. We use the equality of renormalized boundary distances and Santal\'o formula to show the equality of the volumes and to obtain an inequality involving the conformal factor $c$. The result follows from analyzing the equality case of this inequality, and this forces $c\equiv 1$.

Since AH manifolds have infinite volume, we work on the level of the (compact) submanifolds $M_{\eps}=\{\rho>\eps\}$. First, in \S \ref{subsection:metrics_boundary}, we show that the conformal factor is 1 up to a $O(\eps^{\infty})$ remainder near the boundary using ``short geodesics'' as in \cite{GGSU19}, which allows us to show that integrals over the cosphere bundles of the two metrics coincide modulo a small error. We use this together with Santal\'o's formula in \S\ref{subsection:volumes} to prove that the volumes of $M_{\eps}$ for the two metrics coincide modulo $O(\eps^{\infty})$, by using Lefeuvre's diffeomorphism \cite{Lef20}*{\S4} that relates geodesics for the different metrics with the same boundary data. Although this diffeomorphism is not always well defined, we are able to show that this problem can be absorbed as an $O(\eps^{\infty})$ error term. Then we show that the integral of the conformal factor can be bounded from below by the volume plus a $O(\eps^{\infty})$ term; this is the content of \S\ref{subsection:conformal}. Finally, in \S\ref{subsection:rigidity_proof}, we use the AM-GM inequality to obtain that such inequality is indeed an equality, which implies the result by analyzing the equality case.

\subsection{Acknowledgments.} We would like to thank Colin Guillarmou and Thibault Lefeuvre for discussions related to the project and for comments on
an earlier draft of the paper.

The authors were supported by the European Research Council (ERC) under the European Union’s Horizon 2020 research and innovation programme (Grant agreement no. 101162990 -- ADG).

\section{Preliminaries}
In this section, we review some notions which will be used in the proofs of Theorems \ref{theo:main} and \ref{theo:main2}.

\subsection{Unit tangent bundle} We review some properties of the unit cotangent bundles of AH manifolds. We refer to \cite{GGSU19}*{\S 2.1} for a more detailed discussion of the matter.

Let $(\overline M,g)$ be an AH manifold. Given a metric $h
$ on $\partial M$ in the conformal infinity of $g$, there exists (see \cites{GL91, Graham00}) coordinates $(\rho,y)$ of $\overline{M}$ (where $\rho$ is a boundary defining function) such that $|d\rho|_{\rho^{2}g}=1$ \emph{in a neighborhood} of $\partial {M}$, $\rho^{2} g|_{T\partial {M}}=h$, and such that we have $
g=\rho^{-2}(d\rho^{2}+h_\rho),
$ in a collar neighborhood $(0,\eps)_\rho \times \partial{M}$ of $\partial {M}$. Here,  $h_\rho$ is a smooth family of Riemannian metrics on $\partial {M}$ such that $h_{0}=h$. The function $\rho$ is called a \emph{geodesic boundary defining function} associated to $h$.

Let $(\rho,y^{1},\ldots,y^{n})$ be local coordinates near $\partial {M}$ with dual coordinates given by $(\xi_{0}, \eta_{1}, \ldots, \eta_{n})$ on $T^{*}\overline{M}$ so that $T^{*}\overline{M} \ni \xi=\xi_{0} d\rho+\sum_{i=1}^n\eta_{i}dy^{i}$. Let $^bT^{*} \overline{M}$ denote the \emph{$b$-cotangent bundle} of $\overline{M}$. Recall that it is a smooth vector bundle over $\overline{M}$ with basis $\{\rho^{-1} d\rho, dy^{1}, \ldots, d y^{n}\}$, and we use dual coordinates $(\overline{\xi_{0}},\eta)$ so that $\xi=\overline{\xi_{0}}\rho^{-1}d\rho+\sum_{i=1}^n\eta_{i} dy^{i}$. The function $\xi \mapsto \overline{\xi_0}$ is an invariant on $^b T^{*}\overline{M}|_{\partial {M}}$, i.e., it is independent of the choice of coordinates $(\rho, y)$. In particular, the subsets $\{\overline{\xi_{0}}= \pm 1\}$ of $^b T^{*} \overline{M}|_{\partial {M}}$ are invariantly defined independently of any choices.

We denote by $\overline{S^{*}M}=\{(x, \xi) \in {}^b T^{*} \overline{M}:|\xi|_{g}=1\}$ the \emph{unit cosphere bundle} in ${ }^b T^{*} \overline{M}$ with respect to $g$. If we choose a representative metric $h$ in the conformal infinity and use the induced product decomposition near $\partial {M}$ as before, we have, for any $x$ near $\partial {M}$,
\begin{equation}
    \label{eq:SMbar}
\overline{S_{x}^{*}M}=\left\{(x, \xi) : \overline{\xi_{0}}^{2}+\rho^{2}|\eta|_{h_{\rho}}^{2}=1\right\}.
\end{equation}
Note that $\overline{S^{*}M}$ is a smooth non-compact submanifold with boundary of ${ }^b T^{*} \overline{M}$ which can be identified with $S^{*}M$ over $M$. We also define the \emph{influx/outflux boundaries} by
\begin{equation}
\label{eq:influx/ouflux}
\partial_{\pm} S^* M=\{(x, \xi) : x \in \partial {M},\  \overline{\xi_0}=\mp 1\},
\end{equation}
which are invariant and independent of $g$ and of the local coordinates. Note that $\overline{S^{*}M}=S^{*}M \sqcup \partial \overline{S^{*}M}$, where $\partial \overline{S^{*}M}=\partial_{-}S^{*}M \sqcup \partial_{+} S^{*}M$.

Let $X=\tfrac{d}{dt}|_{t=0}\varphi_t\in C^\infty(S^*M,T(S^*M))$ denote the geodesic vector field, and let $\alpha $ be the contact form  satisfying $\alpha(X)=1$ and $\iota_{X} d\alpha=0$. The associated volume form $\mu=\alpha \wedge (d\alpha)^n$ is called the \emph{Liouville measure}. The flow $\varphi_{t} \colon S^{*}M \to S^{*}M$ preserves the Liouville measure since $X$ is the Reeb vector field. 

\subsection{Renormalized boundary distance}
\label{sec:renorm}
Let $(\overline M,g)$ be an AH manifold. By \cite{GGSU19}*{Lemma 2.3}, there exists $\eps>0$ such that for each $(x,\xi) \in S^{*}M$ with $\rho(x)<\eps$, and $\xi=\xi_{0} d\rho+\eta dy$ with $\xi_{0} \leq 0$ (resp. $\xi_{0} \geq 0$), the flow trajectory $\varphi_{t}(x, \xi)$ converges to a point $z_{+} \in \partial_{+} S^{*}M$ (resp. $z_{-} \in \partial_{-}S^{*}M$) as $t \to \infty$ (resp. $t \to -\infty$). In particular, the following map is well defined and smooth 
\begin{equation}
\label{eq:B}
    B_{\pm} \colon S^{*}M\setminus \Gamma_\mp  \to \partial_{\pm}S^{*}M, \quad
    (x,\xi) \mapsto \lim_{t \to \pm \infty} \varphi_{t}(x,\xi).
\end{equation}
Moreover, they extend smoothly to $\overline{S^{*}M}\setminus \overline{\Gamma_\mp}$, where $\overline{\Gamma_\mp}$ is the closure of ${\Gamma_\mp}$ in $\overline{S^{*}M}$. Notice that $B_{\pm}|_{\partial_{\pm}S^{*}M}=\mathrm{id}_{\partial_{\pm}S^*M}$. 

Furthermore, by \cite{GGSU19}*{Lemma 2.1}, the vector field $X$ on $S^* M$ has the form $X=\rho \overline{X}$, where $\overline{X}$ is a smooth vector field on $\overline{S^{*}M}$ that is transverse to the boundary $\partial \overline{S^{*} M}$. Let $(\overline{\varphi}_t)_{t\in \mathbb R}$ denote the flow of $\overline{X}$. The flow lines of $(\overline{\varphi}_t)_{t\in \mathbb R}$ in $S^{*}M$ are the same as the flow lines of $(\varphi_t)_{t\in \mathbb R}$, that is, $(\overline{\varphi}_t)_{t\in \mathbb R}$ is a time reparametrization of the geodesic flow $(\varphi_t)_{t\in \mathbb R}$. More precisely, for all $z \in S^{*}M$ we have $\overline{\varphi}_\tau(z)=\varphi_{t(\tau, z)}(z)$, where
\[
t(\tau, z):=\int_0^\tau \frac{1}{\rho (\overline{\varphi}_{s}(z))} ds.
\]

\begin{defin}
\label{def:B}
Given $z \in \overline{S^{*}M}\setminus \overline{\Gamma_\mp}$, we denote by $\tau_{\pm}(z) \geq 0$ the time so that $\overline{\varphi}_{\pm \tau_{\pm}(z)}(z)=B_{\pm}(z)$. These times are finite and are called \emph{exit times}.     
\end{defin}
We can now define the scattering map.
\begin{defin}
    The \emph{scattering map} is defined by
    \begin{align*}
        S_{g} \colon \partial_{-}S^{*}M \setminus \overline{\Gamma_-} \to \partial_{+}S^{*}M\setminus \overline{\Gamma_+}, \quad
        z \mapsto B_{+}(z)=\overline{\varphi}_{\tau_{+}(z)}(z).
    \end{align*}
\end{defin}
Given $z\in \partial_-S^{*}M\setminus \overline{\Gamma_-}$, we let $\gamma_z$ denote the projection on $\overline{M}$ of the geodesic induced by $z$, i.e.,  $(\gamma_z(s))_{s\in (0,\tau_{+}(z))}=\pi(\overline{\varphi}_{s}(z))_{s \in (0,\tau_{+}(z))}$. We now define the renormalized length of a geodesic.
\begin{defin}
\label{defi:length}
    The \emph{renormalized length} (associated to $\rho$) is given by 
    \begin{align*}
        L_{g} \colon \partial_{-}S^{*}M\setminus \overline{\Gamma_-} \to \mathbb R, \quad
        z \mapsto \lim_{\eps \to 0^{+}} \big( \ell_{g}(\gamma_{z} \cap \{\rho>\eps\})+2\log \eps \big).
    \end{align*}
\end{defin}
The limit is shown to exist in \cite{GGSU19}*{\S 4}. Note that we will also see the renormalized length $L_g$ as a function of $g$-geodesics in the following proofs.

\subsection{Santal\'o formula}
\label{sec:Santalo}
In the course of the proof, we will use Santal\'o's formula which computes the integral of a function in terms of its integrals along geodesics joining boundary points. We refer to \cite{GM26}*{Chapter~4.5} for a detailed discussion of the formula.

More precisely, let $(N,G)$ be a strictly convex, oriented, Riemannian manifold. The metric $G$ defines a \emph{symplectic measure} on $\partial_-S^*N$ which we denote by $\mu_{G,\nu}$. Recall that for any continuous $F \colon \partial_-S^*N\to \mathbb R$, we have 
\begin{equation}
    \label{eq:mu_hnu}
    \int_{\partial_-S^*N}Fd\mu_{G,\nu}=\int_{\partial N}\left(\int_{\partial_-S^*_xN} F(x,\xi)\langle \xi,\nu_G(x)\rangle_Gd\mathrm{Vol}_{S^*_xN}(\xi)\right)d\mu^G_\partial(x),
\end{equation}
where $\nu_G(x)$ is the dual form to the  unit normal vector field to the boundary $\partial N$, $d\mathrm{Vol}_{S^*_xN}$ denotes the Lebesgue measure on the sphere $S_x^*N$ and $d\mu^G_\partial$ is the volume form on $\partial N$ induced by $G|_{\partial N}$. Note in particular that $d\mu_{G,\nu}$ depends only on the value of $G$ on $\partial N$.

Let $\mathcal K^-$ denote the backward trapped set of the geodesic flow $(\varphi_t)_{t\in \mathbb R}$ of $G$. For a continuous function $F \colon S^*N\to \mathbb R$, we have (see \cite{GM26}*{Theorem~4.19}),
\begin{equation}
    \label{eq:Santalo}
    \int_{S^*N\setminus \mathcal K^-} Fd\mu_G=\int_{\partial_-S^*N}\left(\int_0^{\tau_G(\xi)}F(\varphi_t\xi)dt\right)d\mu_{G,\nu},
\end{equation}
where $\tau_G(\xi)$ is the unique time such that $\varphi_{\tau_G(\xi)}(\xi)\in \partial_+S^*N$.

\section{Proof of Theorem \ref{theo:main}}
\label{sec:3}
In order to prove Theorems \ref{theo:main} and \ref{theo:main2}, we will adapt Croke's argument \cite{Croke91} in the compact case to our non-compact setting. More precisely, we use the equality of scattering maps and renormalized lengths to compare the volumes of $g$ and $c^2g$. Since AH manifolds are not compact, we will actually work with the compact submanifolds $M_\eps:=M\cap\{\rho \geq \eps\}$ for any small $\eps>0$. Note that $M_\eps$ is a geodesically convex compact submanifold of $M$. 

In order to apply the results of \cites{GGSU19, Lef20}, it will be more convenient for us to work in the universal cover. Let $\pi \colon \widetilde{\overline M}\to \overline M$ denote the universal cover of $\overline M$. Its boundary $\partial \widetilde{\overline M}$ is a countable union of
connected components which project down through $\pi$ to $\partial M$. Moreover, we have that $\widetilde{\overline M}\setminus \partial \widetilde{\overline M}=\widetilde{M}$ is the universal cover of $M$ and we can see $\partial \widetilde{\overline M}$ as the \emph{visual boundary} of $\widetilde{M}$. This means that given any $\tilde p\neq \tilde q\in \partial \widetilde{\overline M} $, there is a unique $\tilde g$-geodesic $\gamma_{\tilde g}(\tilde p,\tilde q)$ (where $\tilde g$ denotes the lift to $\widetilde{M}$ of $g$) joining $\tilde p$ to $\tilde q$. Moreover, $\gamma_{\tilde g}(\tilde p,\tilde q)$ projects to a $g$-geodesic joining $\pi(\tilde p)$ and $\pi(\tilde q)$. Finally, we note that for any $p\neq q\in \partial M$, there exists a bijective correspondence between $\tilde g$-geodesics joining lifts of $p$ and $q$ and $g$-geodesics joining $p$ and $q$ in a prescribed homotopy class. 

All the constructions recalled in \S \ref{sec:renorm} extend to the universal cover and in particular, we see that the renormalized length of a geodesic $\tilde \alpha$ in $\widetilde{M}$ joining two boundary points is equal to the renormalized length of the projected geodesic $\pi(\tilde \alpha)$ on $M$. In particular, we deduce that the equality of renormalized marked boundary distances of two negatively curved conformal AH metrics $g$ and $c^2g$, or the equality of renormalized boundary distances of two simple AH metrics is equivalent to the equality of renormalized boundary distances of the lifted metrics $\tilde g$ and $\tilde{c}^2\tilde g$ on $\widetilde M$. 

In the next section, we will use this observation and prove both Theorem \ref{theo:main} and Theorem \ref{theo:main2} at the same time. From now on, we will suppress the $\widetilde{\bullet}$ in most notations, except for the universal cover $\widetilde{M}$. We will also write $\partial \widetilde{M}$ instead of $\partial \widetilde{\overline{M}}.$

In the remainder of the paper, we will write $f(\eps)=O(\eps^\infty)$ if for any $N>0$, there is $C_N>0$ such that $|f(\eps)|\leq C_N\eps^N$ when $\eps\to 0^+.$ We will also write $f=O(\rho^\infty)$ when we want to stress that a term is negligible near the boundary.

\subsection{Comparing the two metrics near the boundary} \label{subsection:metrics_boundary}  Since $g$ and $c^2g$ have the same renormalized boundary distances, by \cite{GGSU19}*{Theorem~1.3}, there is a diffeomorphism $\Psi\in C^\infty(\overline{M})$, equal to the identity on $\partial M$ such that the jets of $\Psi^*g$ and $c^2g$ coincide on $\partial M$. In particular, this gives $c|_{\partial M}\equiv 1$. However, for our argument, we will need to show that the full jet of $c$ vanishes near the boundary.

\begin{lemma}
\label{lemm:Sebastian}
Let $(\overline{M},g)$ be an AH manifold, and let $c \in C^{\infty}(\overline{M},\R_{+})$ be such that $(\overline{M},c^{2}g)$ is also an AH manifold. Assume that for some choices $h$ and $h'$ of the conformal representatives of the conformal infinities of $g$ and $c^{2}g$, the renormalized boundary distances are equal. Then $c=1+O(\rho^{\infty})$ near the boundary $\partial M.$
\end{lemma}

\begin{proof}
Let $\rho$ be a geodesic boundary defining function for $g$, associated to $h$. As we already pointed out, we can assume that $c|_{\partial M}\equiv 1$, i.e., $h=h'$. Hence, $\rho$ induces the same conformal representative for $g$ and $c^{2}g$, although it is not a geodesic boundary defining function for the second metric.  By the discussion in \cite{GGSU19}*{p.~2896}, recall that any defining function determines the same renormalized length as the geodesic defining function inducing the same representative for the conformal infinity. Hence, we can, and we will, use $\rho$ to compute both renormalized lengths $L_{g,h}=L_{c^2g,h'}$ even though $\rho $ is not geodesic for $c^2g$.

Now, suppose for a contradiction, that there exist $a \not \equiv 0$ and $m \geq 1$ with
\[
c=1+a(y)\rho^{m}+O(\rho^{m+1}).
\]
If we let $\varphi=\log c$, using the Taylor series $\log(1+u)=u+O(u^{2})$, we can rewrite the previous expansion as
\begin{equation} \label{eq:varphi}
    \varphi=a(y)\rho^m+O(\rho^{m+1}).
\end{equation}
Let us define the family of conformal metrics $g_{s}=e^{2s \varphi}g$, for $s \in [0,1]$. Since $\varphi|_{\partial M}=0$, all $g_{s}$ have the same boundary representative $h$ with respect to the fixed defining function $\rho$. Now, let $\rho_{s}$ be a geodesic boundary defining function for $g_{s}$, associated with $h$. Then $\rho_{s}=e^{\beta_{s}}\rho$, where $\beta_{s}|_{\partial M}=0$ for any $s\in [0,1]$. 

We claim that $\rho_{s}=\rho+O(\rho^{m+1})$, uniformly in $s\in[0,1]$. First, note that the above claim follows from the bound $\beta_{s}=O(\rho^{m})$ uniformly in $s\in[0,1]$, which we prove below. Indeed, since $\rho$ is a geodesic boundary defining function for $g$, we have $\rho_{s}^{2}g_{s}=e^{2(\beta_{s}+s\varphi)}(d\rho^{2}+h_{\rho})$. Moreover, differentiating $\rho_s=e^{\beta_s}\rho$ yields $d\rho_{s}=e^{\beta_{s}}((1+\rho \partial_{\rho}\beta_{s})d\rho+\rho d_{y}\beta_{s})$. This implies that 
\[
|d\rho_{s}|_{\rho_{s}^{2}g_{s}}^{2}=e^{-2s\varphi}((1+\rho \partial_{\rho}\beta_{s})^{2}+\rho^{2}|d_{y}\beta_{s}|_{h_{\rho}}^{2}).
\]
Since $\rho_s$ is a geodesic boundary defining function, we have $|d\rho_{s}|_{\rho_{s}^{2}g_{s}}^{2}=1$ near the boundary. The previous equation and \eqref{eq:varphi} imply
\begin{equation} \label{eq:eikonal}
    1+2a(y)\rho^{m}+O(\rho^{m+1})=(1+\rho \partial_{\rho}\beta_{s})^{2}+\rho^{2}|d_{y}\beta_{s}|_{h_{\rho}}^{2}.
\end{equation}
uniformly in $s$. At this point, let us assume for a contradiction that there exist $k<m$ and $b_{s}$ not identically zero such that $\beta_{s}=b_{s}(y)\rho^{k}+O(\rho^{k+1})$. We write
\[
\rho\partial_{\rho}\beta_{s}=kb_{s}\rho^{k}+O(\rho^{k+1}),  \quad \rho^{2}|d_{y}\beta_{s}|_{h_{\rho}}^{2}=O(\rho^{2k+2}).
\]
Hence, after canceling the constant terms, the left-hand side of \eqref{eq:eikonal} is of order $\rho^{m}$, while the right-hand side is of order $\rho^{k}$ and this gives the desired contradiction. In other words, we have shown that $\beta_{s}=O(\rho^{m})$, from which it follows that $\rho_{s}=\rho+O(\rho^{m+1})$. Note that equivalently, we have $\rho=\rho_{s}+O(\rho_{s}^{m+1})$. Together with \eqref{eq:varphi}, this implies
\begin{equation} \label{eq:varphi2}
    \varphi=a(y)\rho_{s}^{m}+O(\rho_{s}^{m+1})
\end{equation}
The idea of the proof is to obtain a contradiction from \eqref{eq:varphi2} using the equality of renormalized lengths. More precisely, we fix $y_{0} \in \partial M$ and $\omega_{0} \in T_{y_{0}}^{*}\partial M$ with $|\omega_{0}|_{h}^{2}=1$, and set $\eta_{0}=\delta^{-1}\omega_{0}$ (with $\delta>0$ small). Let $\gamma_{s,\delta}$ be the $g_{s}$-geodesic whose incoming boundary covector is $z=(y_{0}, \delta^{-1} \omega_{0})$. Since $\delta>0$ is small, this defines a ``small geodesic'' which we will now study following the ideas developed in \cite{GGSU19}*{\S 2.2}.

Near the boundary, we use the coordinates $(\rho_s,y_s,\overline{\xi_0},\eta)$ recalled in \eqref{eq:SMbar}. In the region where  $|\eta|_{h_{\rho_s}}>1/2$, we use the  diffeomorphism $\theta_{s} \colon [0,\tau_{+,s}(z)] \to [0,\pi]$ defined by $\sin (\theta_{s})=\rho_{s}|\eta|_{h_{\rho_{s}}}$ and $\cos(\theta_{s})=\overline{\xi_0}$, see \cite{GGSU19}*{\S 2.2} for more details. Using coordinates $(\theta_s,y_s,\overline{\xi_0},\eta)$ and because all $g_s$ have the same leading boundary metric $h$, these short geodesics satisfy, uniformly in $s\in [0,1]$,
\begin{equation} \label{eq:shortgeo1}
    \rho_{s}(\theta_{s})=\delta \sin( \theta_{s})+O(\delta^{2}), \quad y_{s}(\theta_s)=y_{0}+\delta(1-\cos( \theta_{s})) \omega_{0}^{\sharp}+O(\delta^{2}),
\end{equation}
where $\omega_{0}^{\sharp}$ is the vector associated with $\omega_{0}$ by the musical isomorphism through $h(y_{0})$. For the proof of the first equality, see \cite{GGSU19}*{Lemma 2.8} and for the second one, see the expression of $u=(y_s-y_0)/\delta $ computed in the proof of \cite{GGSU19}*{Lemma 2.8} . From the proof of \cite{GGSU19}*{Proposition 3.15}, we also have
\begin{equation} \label{eq:shortgeo2}
    dt_{g_{s}}=\frac{(1+O(\delta))}{\sin (\theta_{s})}d\theta_{s},
\end{equation}
where $t_{g_s}$ is the time parameter of the geodesic flow of $g_s.$
Note that \eqref{eq:varphi2} and \eqref{eq:shortgeo1} imply that along $\gamma_{s,\delta}$, we have
\begin{equation} \label{eq:varphi3}
    \varphi=a(y_{0})\delta^{m}\sin^{m}(\theta_{s})+O(\delta^{m+1}),
\end{equation}
uniformly in $s$. Here we used the fact that the $y$ coordinate along the geodesic $\gamma_{s,\delta}$ is constant equal to $y_0.$ Recall that we use the same $\rho$ to compute the variations of the length of $\gamma_{s,\delta}$. The variation of renormalized length is given by the \emph{X-ray transform} of the variation of metrics, see \cite{GGSU19}*{p. 2902},
$$\frac{d}{ds}L_{g_{s},\rho}(y_{0},\delta^{-1}\omega_{0})=I_0^s(\varphi)(y_{0},\delta^{-1}\omega_{0})=\int_{\gamma_{s,\delta}}\varphi dt_{g_{s}}. $$
Using equations \eqref{eq:shortgeo2} and \eqref{eq:varphi3}, this gives
\begin{equation} \label{eq:variation}
    \frac{d}{ds}L_{g_{s},\rho}(y_{0},\delta^{-1}\omega_{0})=\int_{\gamma_{s,\delta}}\varphi dt_{g_{s}}=a(y_{0})\delta^{m}\int_{0}^{\pi} \sin^{m-1}(\theta_{s})d\theta_{s}+O(\delta^{m+1}).
\end{equation}
By \cite{GGSU19}*{Proposition~5.4}, the equality of the renormalized boundary distances imply $L_{g,\rho}=L_{c^2g,\rho}$. Now,
since the equation above holds uniformly in $s\in [0,1]$, integrating \eqref{eq:variation} from $0$ to $1$, we obtain
\[
0=L_{g,\rho}(y_{0},\delta^{-1}\omega_{0})-L_{c^2g,\rho}(y_{0},\delta^{-1}\omega_{0})=a(y_{0})\int_{0}^{\pi} \sin^{m-1}\theta_{s}d\theta_{s}\delta^{m}+O(\delta^{m+1}).
\]
Since $\int_{0}^{\pi} \sin^{m-1}(\theta_{s})d\theta_{s}\neq 0$, this implies
 that $a(y_{0})=0$. But $y_{0}$ was arbitrary so this shows that $a$ vanishes identically, giving the desired contradiction. This concludes the proof and we have shown that $c=1+O(\rho^{\infty})$ near $\partial M.$  
\end{proof}

Let $\eps>0$. For a metric $G$ on $M$, we denote by $d\mu_\partial^{G,\eps}$ the Riemannian volume on $\partial {M_\eps}$ defined by the induced metric $G|_{\partial {M_\eps}}$. 
\begin{lemma}
\label{lemm:eps}  
One has $d\mu_\partial^{g,\eps}=d\mu_\partial^{c^2g,\eps}+O(\eps^\infty)$ when $\eps\to 0^{+}$. That is, for any continuous function $G \colon \partial M_\eps\to \mathbb R$, one has
$$\int_{\partial M_\eps}G(x)d\mu_\partial^{c^2g,\eps}(x)=\int_{\partial M_\eps}G(x)d\mu_\partial^{g,\eps}(x)+\|G\|_{C^0}O(\eps^\infty). $$
    
\end{lemma}
\begin{proof}
   By Lemma \ref{lemm:Sebastian}, the jets of the metrics $g$ and $c^2g$ coincide at $\partial M$. In particular, for $\eps>0$, one has $g|_{\partial M_\eps}=c^2g|_{\partial M_\eps}+O(\eps^\infty)$ from which the lemma follows.
\end{proof}
 Next, we note that for any $x\in M$, there is a natural fiberwise rescaling:
\begin{equation}
    \label{eq:psi}
    \Psi_x \colon (S_g^*)_xM\to (S_{c^2g}^*)_xM,\quad \Psi_x(\xi)=c(x)\xi.
\end{equation}
It is easy to see that for any $x\in M$, we have $\Psi_x^*(d\mathrm{Vol}_{(S_{c^2g}^*)_xM})=d\mathrm{Vol}_{(S_{g}^*)_xM}$. For $\eps>0$ and for a metric $G$ on $M$, let $\nu_{G}^\eps(x)$ denote the dual form to the unit normal vector to $\partial {M_\eps}$ for the metric $G$. We have for any
$ x\in \partial {M_\eps},$ that  $ c(x)^{-1}\nu_{g,\eps}(x)= \nu_{c^2g,\eps}(x). $
This implies that 
$$\forall x\in \partial {M_\eps}, \ \forall \xi\in (S_{c^2g}^*)_xM, \quad \langle \Psi_x^{-1}(\xi),\nu_{g,\eps}(x)\rangle_{g}= \langle \xi,\nu_{c^2g,\eps}(x)\rangle_{c^2g}. $$
In particular, for any continuous function $F \colon (S_{c^2g}^*)_xM_\eps\to \mathbb R$ and any $x\in \partial {M_\eps}$,
\begin{align*}
\int_{\partial_-(S_{c^2g}^*)_x{M_\eps}}& F(x,\xi)\langle \xi ,\nu_{c^2g,\eps}(x)\rangle_{c^2g}d\mathrm{Vol}_{(S_{c^2g}^*)_xM}(\xi)
\\&=\int_{\partial_-(S_{g}^*)_x{M_\eps}} F(x,\Psi_x^{-1}\eta)\langle \eta,\nu_{g,\eps}(x)\rangle_gd\mathrm{Vol}_{(S_{g}^*)_xM}(\eta).
\end{align*}
For a metric $G$ on $M$ and $\eps>0$, we denote by $d\mu_{G,\eps}$ the symplectic measure on $\partial_-S^{*}{M_\eps}$ defined in \S \ref{sec:Santalo}.
Together with Lemma \ref{lemm:eps}, we have shown:
\begin{prop}
    \label{prop:sympl} We have, when $\eps \to 0^{+}$, $d\mu_{g,\eps}=d\mu_{c^2g,\eps}+O(\eps^\infty)$. That is, for any continuous function $F \colon \partial_-S^*_{c^2g}{M}_\eps\to \mathbb R$, one has
\begin{equation*}
\int_{\partial_-S^*_{c^2g}{M}_\eps}F(x,\xi)d\mu_{c^2g,\eps}(x,\xi)=\int_{\partial_-S^*_{g}{M}_\eps}F(x,\Psi_x^{-1}\xi)d\mu_{g,\eps}(x,\xi)+\|F\|_{C^0}O(\eps^\infty).
\end{equation*}
\end{prop}

\subsection{Comparing the volumes of \texorpdfstring{$M_\eps$}{Mε} for \texorpdfstring{$g$}{g} and \texorpdfstring{$c^2g$}{c²g}} \label{subsection:volumes}
In this subsection, we prove the following Proposition.

\begin{prop} 
One has 
\begin{equation}
\label{eq:volume}
\mathrm{Vol}_g(M_\eps)=\mathrm{Vol}_{c^2g}(M_\eps)+O(\eps^\infty),
\end{equation}
when $\eps\to 0^+$.
\end{prop}
We remark that the previous proposition could be extracted from the proof of \cite{Lef20}*{Lemma~4.8} and we provide the argument here for completeness.
\begin{proof} As explained in the beginning of \S \ref{sec:3}, we will work in the universal cover in order to apply some of the results of \cites{GGSU19, Lef20}.

For $\eps>0$, we let $\widetilde{M_\eps}$ denote the lift to the universal cover of $M_\eps$. 
    Given $(x,\xi)\in S_g^*\widetilde{M}\setminus (\widetilde{\Gamma^g_-}\cup \widetilde{\Gamma^g_+})$, recall that the geodesic generated by $(x,\xi)$ has endpoints $(p,q)\in \partial \widetilde{M}\times  \partial \widetilde{M}$. We follow the idea of \cite{Lef20}*{\S 4} and build a correspondence from $\partial_-S^*_g\widetilde{M_\eps}$ to $\partial_-S^*_{c^2g}\widetilde{M_\eps}$ using the intersection of the $c^2g$-geodesic with endpoints $(p,q)$ and $\{\rho=\eps\}$. 

Given $(\tilde x,\tilde\xi)\in \partial_-S^*_g\widetilde{M_\eps}$, we represent $\tilde\xi$ by the angle $\omega\in[0,\pi]$ such that $\sin (\omega)=|\tilde g(\tilde \nu_{g,\eps}(\tilde x),\tilde \xi)|$, where $\tilde \nu_{g,\eps}(\tilde x)$ stands for the lift of $\nu_{g,\eps}(x)$. The construction above might not be well defined for small angles $\omega$. Since we will compute integrals, it will be sufficient to show that the set of ``bad angles'' has negligible measure. The following result is due to Lefeuvre, see \cite{Lef20}*{Lemma~4.3}.
\begin{lemma}
\label{lemm:angle}
Let $N\in \mathbb N$. Then for any $(\tilde x,\tilde\xi(\omega))\in \partial_-S^*_g\widetilde{M_\eps}\setminus \widetilde{\Gamma_-^g}$ with $\omega \in [\eps^N,\pi-\eps^N]$, if we denote by $\gamma_g(p,q)$ the $g$-geodesic defined by $(\tilde x,\tilde\xi(\omega))$ with endpoints $(p,q)\in \partial \widetilde{M}\times \partial \widetilde{M}$, then the $c^2g$-geodesic $\gamma_{c^2g}(p,q)$ with endpoints $(p,q)$ intersect $\{\rho=\eps\}$, see \cite{Lef20}*{Figure 3}.
\end{lemma}
Fix some large $N\in \mathbb N$ to be determined later.
Let $\mathcal U_\eps$ the subset of $\big(\partial_-S^*_g\widetilde{M_\eps}\setminus \widetilde{\Gamma_-^g}\big)$ with angles in $[\eps^N,\pi-\eps^N]$. We define a mapping $\tilde{\psi}_\eps \colon \mathcal U_\eps \to \partial_-S^*_{c^2g}\widetilde{M_\eps}$ that sends $(\tilde x,\tilde \xi)\in \mathcal U_\eps$ to the first intersection point (with corresponding covector) given by Lemma \ref{lemm:angle}. Since the scattering maps coincide (\cite{Lef20}*{Proposition~4.2}) and since the jets of the metrics are the same on the boundary, $\psi_\eps$ is approximately equal to the identity on $\mathcal U_\eps$ when $\eps\to 0^{+}$. Indeed, by \cite{Lef20}*{Lemma~4.4},
\begin{equation}
\label{eq:Id}
\|\tilde{\psi}_\eps-\mathrm{Id}\|_{C^1}=O(\eps^\infty).
\end{equation}
 For $(\tilde x,\tilde\xi(\omega))\in \partial_-S^*_{g}\widetilde{M_\eps}$ (resp. $(\tilde x,\tilde\xi(\omega))\in \partial_-S^*_{c^2g}\widetilde{M_\eps}$) we let $\tilde\ell_g^\eps(\tilde x,\tilde \xi)$ (resp. $\tilde\ell_{c^2g}^\eps(\tilde x,\tilde\xi)$) denote the length of the $g$ (resp. $c^2g$) geodesic defined by $(\tilde x,\tilde \xi)$ in $\widetilde{M_\eps}$. Since the renormalized lengths coincide,  $\tilde\ell_g^\eps(\tilde x,\tilde\xi)$ and $\tilde\ell_{c^2g}^\eps(\psi_\eps(\tilde x,\tilde\xi))$ are approximately equal when $\eps\to 0^{+}$. More precisely, by \cite{Lef20}*{Lemma~4.5},
\begin{equation}
\label{eq:length}
\sup_{(\tilde x,\tilde \xi)\in \mathcal U_\eps}|\tilde\ell_g^\eps(\tilde x,\tilde\xi)-\tilde\ell_{c^2g}^\eps(\psi_\eps(\tilde x,\tilde\xi))|=O(\eps^\infty).
\end{equation}
Using Santal\'o's formula and \eqref{eq:length}, we show that the volumes of $M_\eps$ for $g$ and $c^2g$ are equal up  to a negligible term when $\eps\to 0^{+}$.

In all the following computations, we will ignore the trapped set since it is either empty in the simple case, or its volume is zero in negative curvature (see \cite{GGSU19}*{Equation~(2.21)}), and hence does not change the value of any of the integrals we will compute.
Applying Santal\'o's formula \eqref{eq:Santalo} to the constant function equal to $1$, we obtain
$$\mathrm{Vol}_{c^2g}(M_\eps) \mathrm{Vol}(\mathbb S^{n-1})=\int_{\partial_-S^*_{c^2g}{M}_\eps}\ell_{c^2g}^\eps(x,\xi)d\mu_{c^2g,\eps}(x,\xi).$$

In particular, we see that Proposition \ref{prop:sympl} implies
$$\mathrm{Vol}_{c^2g}(M_\eps) \mathrm{Vol}(\mathbb S^{n-1})=\int_{\partial_-S^*_{g}{M}_\eps}\ell_{c^2g}^\eps(x,\xi)d\mu_{g,\eps}(x,\xi)+O(\eps^\infty),$$
where we used the uniform bound $\ell_{c^2g}^\eps(x,\xi)=O(\eps^{-1})$ to absorb the length in the remainder. 
Next, we denote by $\mathcal V_\eps:=\tilde \psi_\eps(\mathcal U_\eps)$. Let $\pi \colon S_g^*\widetilde M\to S_g^*M$ be the natural projection. We obtain
\begin{align*}\mathrm{Vol}_{c^2g}(M_\eps) \mathrm{Vol}(\mathbb S^{n-1})=&\int_{\pi(\mathcal V_\eps)}\ell_{c^2g}^\eps(x,\xi)d\mu_{g,\eps}(x,\xi)
\\&+\int_{\partial_-S^*_{g}{M}_\eps\setminus \pi(\mathcal V_\eps )}\ell_{c^2g}^\eps(x,\xi)d\mu_{g,\eps}(x,\xi)+O(\eps^\infty).
\end{align*}
Notice that the volume of $M_\eps$ grows polynomially in $\eps^{-1}$. Moreover, according to Lemma \ref{lemm:angle}, the support of the integral in each fiber for the second term is $O(\eps^N)$ for any $N\geq 1$. In total, the second term is negligible:
$$ \int_{\partial_-S^*_{g}{M}_\eps\setminus \pi(\mathcal V_\eps)}\ell_{c^2g}^\eps(x,\xi)d\mu_{g,\eps}(x,\xi)=O(\eps^\infty).$$
For the first term, we lift the computation to the universal cover. Remark that for any $(x,\xi)\in \partial_-S^*_{g}{M}_\eps$ and any lift $(\tilde x,\tilde \xi)\in \partial_-S^*_{g}\widetilde{M_\eps}$, we have $\ell^\eps_g(x,\xi)=\tilde \ell^\eps_{g}(\tilde x,\tilde \xi)$, and a similar statement for $c^2g$.  Let $\mathcal W_\eps\subset \partial_-S^*_{g}\widetilde{M_\eps} $ be such that $\pi \colon \mathcal W_\eps \to \pi(\mathcal V_\eps)$ is a diffeomorphism. In particular, using \eqref{eq:Id} and \eqref{eq:length}, we obtain
\begin{align*}
\int_{\pi(\mathcal V_\eps)}&\ell_{c^2g}^\eps(x,\xi)d\mu_{g,\eps}(x,\xi)=\int_{\mathcal W_\eps}\tilde\ell_{c^2g}^\eps(\tilde x,\tilde \xi)d\mu_{g,\eps}(\tilde x,\tilde \xi)
\\&
=\int_{\psi_\eps^{-1}(\mathcal W_\eps)}\tilde \ell_{g}^\eps(\tilde y,\tilde \eta)d[(\psi_\eps)^*(\mu_{g,\eps})](\tilde y,\tilde \eta)=\int_{\psi_\eps^{-1}(\mathcal W_\eps)}\tilde \ell_{g}^\eps(\tilde y,\tilde \eta)d\mu_{g,\eps}(\tilde y,\tilde \eta)+O(\eps^\infty)
\\&=\int_{\pi\circ\psi_\eps^{-1}(\mathcal W_\eps)} \ell_{g}^\eps( y, \eta)d\mu_{g,\eps}( y, \eta)+O(\eps^\infty)=\int_{\mathcal U_\eps} \ell_{g}^\eps( y, \eta)d\mu_{g,\eps}( y, \eta)+O(\eps^\infty)
\\&=\int_{\partial_-S^*_{g}{M}_\eps}\ell_{g}^\eps(x,\xi)d\mu_{g,\eps}(x,\xi)-\int_{\partial_-S^*_{g}{M}_\eps\setminus \mathcal U_\eps}\ell_{g}^\eps(x,\xi)d\mu_{g,\eps}(x,\xi)+O(\eps^\infty)
\\&=\mathrm{Vol}_{g}(M_\eps) \mathrm{Vol}(\mathbb S^{n-1})+O(\eps^\infty),
\end{align*}
where we used Santal\'o's formula \eqref{eq:Santalo} for the first term and the same argument as before to treat the second one. Combining everything gives \eqref{eq:volume}.
\end{proof}
\subsection{Bounding the integral of the conformal factor} \label{subsection:conformal} In this subsection, we show the following proposition.
\begin{prop}
One has
\begin{equation}
\label{eq:lower}
\int_{M_\eps}c(x)d\mu_g(x)\geq \mathrm{Vol}(M_\eps)+O(\eps^{\infty}).
\end{equation}
\end{prop}
\begin{proof} In the following, we denote by $\gamma_G(x,\xi)$ the $G$-geodesic generated by $(x,\xi)$.  
Applying Santal\'o's formula \eqref{eq:Santalo} for the constant function, we obtain
\begin{align*}
\mathrm{Vol}_{c^2g}(M_\eps) \mathrm{Vol}(\mathbb S^{n-1})&=\int_{\partial_-S^*_{c^2g}{M}_\eps}\ell_{c^2g}^\eps(x,\xi)d\mu_{c^2g,\eps}(x,\xi)
\\&= \int_{\partial_-S^*_{c^2g}{M}_\eps}\ell_{c^2g}\big(\gamma_{c^2g}(x,\xi)\cap \{\rho \geq \eps\} \big)d\mu_{c^2g,\eps}(x,\xi).
\end{align*}
Let $p^\eps_{x,\xi}$ denote the geodesic path between the endpoints of $\gamma_{c^2g}(x,\xi)\cap \{\rho \geq \eps\}$ and of $\gamma_{g}(x,\xi)\cap \{\rho \geq \eps\}$ not equal to $(x,\xi)$, see Figure \ref{fig:segment}.
\begin{figure}
    \centering
\begin{tikzpicture}[x=0.75pt,y=0.75pt,yscale=-1,xscale=1]
\draw    (247,42) .. controls (298,71) and (349,168) .. (339,270) ;
\draw    (331,42) .. controls (382,71) and (433,168) .. (423,270) ;
\draw    (424.55,249.69) .. controls (264,216.23) and (239,188.23) .. (240,152.23) .. controls (241,116.23) and (306,73.23) .. (344,51.23) ;
\draw [color={rgb, 255:red, 0; green, 27; blue, 255 }  ,draw opacity=1 ]   (423.55,232) .. controls (383,236.23) and (254.1,225.77) .. (234.55,175) .. controls (215,124.23) and (341.45,78.14) .. (362,69.23) ;
\draw  [color={rgb, 255:red, 0; green, 0; blue, 0 }  ,draw opacity=1 ][fill={rgb, 255:red, 0; green, 0; blue, 0 }  ,fill opacity=1 ] (335,228) .. controls (335,225.41) and (337.04,223.31) .. (339.55,223.31) .. controls (342.06,223.31) and (344.1,225.41) .. (344.1,228) .. controls (344.1,230.59) and (342.06,232.69) .. (339.55,232.69) .. controls (337.04,232.69) and (335,230.59) .. (335,228) -- cycle ;
\draw    (339.55,228) -- (322.94,223.73) ;
\draw [shift={(321,223.23)}, rotate = 14.41] [color={rgb, 255:red, 0; green, 0; blue, 0 }  ][line width=0.75]    (10.93,-3.29) .. controls (6.95,-1.4) and (3.31,-0.3) .. (0,0) .. controls (3.31,0.3) and (6.95,1.4) .. (10.93,3.29)   ;
\draw [color={rgb, 255:red, 255; green, 0; blue, 31 }  ,draw opacity=1 ][line width=1.5]    (292,85.23) -- (300.14,97.63) ;
\draw [color={rgb, 255:red, 255; green, 0; blue, 31 }  ,draw opacity=1 ]   (400,97.23) .. controls (341.29,77.57) and (387.53,113.87) .. (313.13,96.5) ;
\draw [shift={(312,96.23)}, rotate = 13.32] [color={rgb, 255:red, 255; green, 0; blue, 31 }  ,draw opacity=1 ][line width=0.75]    (10.93,-3.29) .. controls (6.95,-1.4) and (3.31,-0.3) .. (0,0) .. controls (3.31,0.3) and (6.95,1.4) .. (10.93,3.29)   ;
\draw  [color={rgb, 255:red, 0; green, 0; blue, 0 }  ,draw opacity=1 ][fill={rgb, 255:red, 0; green, 0; blue, 0 }  ,fill opacity=1 ] (420,250.69) .. controls (420,248.1) and (422.04,246) .. (424.55,246) .. controls (427.06,246) and (429.1,248.1) .. (429.1,250.69) .. controls (429.1,253.27) and (427.06,255.37) .. (424.55,255.37) .. controls (422.04,255.37) and (420,253.27) .. (420,250.69) -- cycle ;
\draw  [color={rgb, 255:red, 0; green, 27; blue, 255 }  ,draw opacity=1 ][fill={rgb, 255:red, 0; green, 27; blue, 255 }  ,fill opacity=1 ] (420,231.31) .. controls (420,228.73) and (422.04,226.63) .. (424.55,226.63) .. controls (427.06,226.63) and (429.1,228.73) .. (429.1,231.31) .. controls (429.1,233.9) and (427.06,236) .. (424.55,236) .. controls (422.04,236) and (420,233.9) .. (420,231.31) -- cycle ;

\draw (265,145) node [anchor=north west][inner sep=0.75pt]   [align=left] {$\displaystyle \gamma_{g}(x,\xi)$};
\draw (237,232) node [anchor=north west][inner sep=0.75pt]  [color={rgb, 255:red, 0; green, 27; blue, 255 }  ,opacity=1 ] [align=left] {$\displaystyle \gamma _{c^{2}g}(x,\xi)$};
\draw (433,125) node [anchor=north west][inner sep=0.75pt]   [align=left] {$\displaystyle \partial M$};
\draw (201,51) node [anchor=north west][inner sep=0.75pt]   [align=left] {$\displaystyle \{\rho =\eps \}$};
\draw (341.55,231) node [anchor=north west][inner sep=0.75pt]   [align=left] {$\displaystyle x$};
\draw (317,203) node [anchor=north west][inner sep=0.75pt]   [align=left] {$\displaystyle \xi $};
\draw (406,85) node [anchor=north west][inner sep=0.75pt]  [color={rgb, 255:red, 255; green, 0; blue, 31 }  ,opacity=1 ] [align=left] {$\displaystyle p_{x,\xi }^{\eps}$};
\draw (435,202) node [anchor=north west][inner sep=0.75pt]  [color={rgb, 255:red, 0; green, 27; blue, 255 }  ,opacity=1 ] [align=left] {$\displaystyle B_{-}^{c^{2} g}( x,\xi )$};
\draw (441,239) node [anchor=north west][inner sep=0.75pt]   [align=left] {$\displaystyle B_{-}^{g}( x,\xi )$};

\end{tikzpicture}
    \caption{The segment $p_{x,\xi}^{\eps}$.}
    \label{fig:segment}
\end{figure}
Note that 
\begin{equation}
    \label{eq:p_x}
\ell_{c^2g}(p^\eps_{x,\xi})=O(\eps^\infty). 
\end{equation}
Indeed, let $(\tilde x,\tilde \xi)$ be a lift of $(x,\xi)$ to $\partial_-S^*_{g}\widetilde{M}_\eps$. Let $B_-^g(\tilde x,\tilde \xi)$ (resp. $B_-^{c^2g}(\tilde x,\tilde \xi)$) be as in \eqref{eq:B}. We claim that the distance between $B_-^g(\tilde x,\tilde \xi)$ and $B_-^{c^2g}(\tilde x,\tilde \xi)$ is $O(\eps^\infty)$. Indeed, by \cite{GGSU19}*{Equation~(2.18)}, we have $\tau_-^g(\tilde x,\tilde \xi)\sim -\eps$ and $\tau_-^{c^2g}(\tilde x,\tilde \xi)\sim -\eps$ when $\eps\to 0^+$. Next, we use Lemma \ref{lemm:Sebastian} to obtain that $\|\overline{\varphi}^{g}_\tau-\overline{\varphi}^{c^2g}_\tau\|_{C^k}=O(\tau^\infty)$, see for instance \cite{Lef20}*{Remark~4.1}, and these two facts together with Definition \ref{def:B} imply the claim. 

 Let us denote by $z_{g,\eps}(\tilde x,\tilde \xi)$ (resp. $z_{c^2g,\eps}(\tilde x,\tilde \xi)$) the second intersection point of $\gamma_g(\tilde x,\tilde \xi) $ (resp. $\gamma_{c^2g}(\tilde x,\tilde \xi)$) with $\{\rho=\eps\}$. We see that for $G=g,c^2g$, we have
$$z_{G,\eps}(\tilde x,\tilde \xi)= \overline{\varphi}^G_{-\tau_+(z_{G,\eps}(\tilde x,\tilde \xi))}\big(S_G(B_-^G(\tilde x,\tilde \xi))\big).$$
Now, since the boundary distances of $g$ and $c^2g$ coincide, their scattering maps also coincide, see for instance \cite{Lef20}*{Proposition~4.2}. In particular, using the smoothness of the scattering maps, we deduce that $S_g(B_-^g(\tilde x,\tilde \xi))$ and $S_{c^2g}(B_-^{c^2g}(\tilde x,\tilde \xi))$ are $O(\eps^\infty)$-close. 
Using the same argument as above, we deduce \eqref{eq:p_x}. Together with Lemma \ref{lemm:eps}, this implies, using the fact that geodesics minimize length,
\begin{align*}
\mathrm{Vol}_{c^2g}(M_\eps)& \mathrm{Vol}(\mathbb S^{n-1})\leq \int_{\partial_-S^*_{c^2g}{M}_\eps}\ell_{c^2g}\big(\big(\gamma_{g}(x,\xi)\cap \{\rho \geq \eps\} \big)\cup p^\eps_{x,\xi}\big)d\mu_{c^2g,\eps}(x,\xi)
\\&=\int_{\partial_-S^*_{g}{M}_\eps}\left(\int_0^{\ell_g^\eps(x,\xi)}c(\varphi_t^g(x,\xi))dt\right)d\mu_{g,\eps}(x,\xi)+O(\eps^\infty)
\\&=\int_{S^*_gM_\eps}c(x)d\mu_g(x,\xi)+O(\eps^\infty)=\int_{M_\eps}c(x)d\mu_g(x)\mathrm{Vol}(\mathbb S^{n-1})+O(\eps^\infty),
\end{align*}
where we applied Santal\'o's formula \eqref{eq:Santalo} for the function $c$.
This concludes the proof of the proposition.
\end{proof}
\subsection{Obtaining rigidity} \label{subsection:rigidity_proof}
We now conclude the proof of Theorems \ref{theo:main} and \ref{theo:main2}.
 In the compact case, this follows from H\"older's inequality once the equality of volumes is established. The fact that $c\equiv 1$ then follows from the equality case in this inequality. Since all our estimates hold modulo $O(\eps^\infty)$, we will use a slightly different argument to finish the proof and use the AM-GM inequality instead of H\"older's inequality to take advantage of the non-negativity of the integrand. 
\begin{proof}[Proof of Theorems \ref{theo:main} and \ref{theo:main2}]
We apply the AM-GM inequality to obtain
\begin{equation}
    \label{eq:pos}
    c\leq \frac{1}{n}( c^n+\underbrace{1+\ldots+1}_{(n-1) \text{ times }})=\frac{1}{n}c^n+\frac{n-1}{n}.
\end{equation}
Integrating over $M_\eps$ gives
\begin{align*}
\frac 1n \mathrm{Vol}_{c^2g}(M_\eps)+\frac{n-1}n \mathrm{Vol}_{g}(M_\eps)-\int_{M_\eps} cd\mu_g=\int_{M_\eps}\underbrace{\left(\frac{1}{n}c^n+\frac{n-1}{n}-c\right)}_{\geq 0}d\mu_g\geq 0. 
\end{align*}
Using \eqref{eq:volume}, this implies
$\int_{M_\eps} cd\mu_g(v)\leq \mathrm{Vol}_g(M_\eps)+O(\eps^\infty).$
Combining this with \eqref{eq:lower} we obtain
$\int_{M_\eps} cd\mu_g(v)=\mathrm{Vol}_g(M_\eps)+O(\eps^\infty)$. Hence
$$\int_{M_\eps}{\left(\frac{1}{n}c^n+\frac{n-1}{n}-c\right)}d\mu_g(x)=O(\eps^\infty).$$ 
Fix a small $\eps_{0}>0$. Note that $M_{\eps_{0}} \subset M_{\eps}$ for any $0<\eps<\eps_{0}$. Since the integrand is non-negative,
$$\mathcal{I} \colon \eps \mapsto \int_{M_\eps}{\left(\frac{1}{n}c^n+\frac{n-1}{n}-c\right)}d\mu_g(x), $$
is non-increasing. Furthermore $0 \leq \mathcal{I}(\eps_{0}) \leq \mathcal{I}(\eps)=O(\eps^{\infty})$. It follows that $\mathcal{I}$ constant equal to $0$, that is, for any $\eps>0$, one has 
$$\int_{M_\eps}{\left(\frac{1}{n}c^n+\frac{n-1}{n}-c\right)}d\mu_g(x)=0.$$
This in turn implies that the integrand vanishes identically on $M_\eps$ and thus on $M$. The equality case in \eqref{eq:pos} forces $c\equiv 1$ on $M$, and this shows that $g=c^2g$.
\end{proof}

\end{document}